\documentclass[12pt,twoside]{amsart}
\usepackage{amsmath, amsthm, amscd, amsfonts, amssymb, graphicx}
\usepackage{enumerate}
\usepackage[colorlinks=true,
linkcolor=blue,
urlcolor=cyan,
citecolor=red]{hyperref}
\usepackage{mathrsfs}
\newtheorem{theorem}{Theorem}[section]

\newtheorem{remark}{Remark}[section]

\newtheorem{corollary}{Corollary}[section]

\numberwithin{equation}{section}

\begin{document}
	
\title{Mathematical inequalities on some weighted means}
\author{Shigeru Furuichi, Kenjiro Yanagi and Hamid Reza Moradi}

\begin{abstract}
Some mathematical inequalities among various weighted means are studied. Inequalities on weighted logarithmic mean are given. Besides, the gap in Jensen's inequality is studied as a convex function approach. Consequently, some non-trivial inequalities on means are given. Some operator inequalities are also shown.
\end{abstract}
\maketitle
\pagestyle{myheadings}
\markboth{\centerline {}}
{\centerline {}}
\bigskip
\bigskip
\section{Introduction}
The weighted arithmetic-geometric mean inequality (Young inequality) states that
$ (1-v)a+vb\geq a^{1-v}b^v$ for $0\le v\le 1$ and $a,b>0,$
which is equivalent to $(1-v)+v t \geq t^v$ for $0\le v\le 1$ and $t>0$.
$(1-v)+v t$ and  $t^v$ are often called representing functions of the weighted arithmetic mean $a\nabla_vb:=(1-v)a+vb$ and the weighted geometric mean $a\sharp_vb:=a^{1-v}b^{v}$, respectively. We easily find the following relation  for $0\le v\le 1$ and $t>0$:
\begin{equation}\label{sec1_ineq01}
 \left\{(1-v)+vt^{-1}\right\}^{-1} \leq t^v \leq (1-v)+v t,
\end{equation}
where $\left\{(1-v)+vt^{-1}\right\}^{-1}$ is a representing function of  the weighted harmonic mean $a!_vb:=\left\{(1-v)a^{-1}+vb^{-1}\right\}^{-1}$ for  $0\le v\le 1$ and $a,b>0$.
In this paper, we use the representing functions for the corresponding means, to simplify the results.
Refinements and reverses for the above inequalities have been studied in many literature. We  refer the readers to \cite[Section 2]{FM2020}. 

Here, we summarize some reverse inequalities \eqref{sec1_ineq01} with some ratios.

\begin{itemize}
\item[(i)] We have the following relations
\begin{equation}\label{sec1_ineq02}
\max_{0\le v \le 1}\frac{(1-v)+vt}{t^v}=S(t),\quad \max_{0\le v \le 1}\frac{t^v}{\left\{(1-v)+vt^{-1}\right\}^{-1}}=S(t),
\end{equation}
where $S(t):=\dfrac{ t^{1/t-1}}{e\log t^{1/t-1}}$ is called Specht ratio \cite[Section 2]{FM2020}. 
\item[(ii)] We also have the following relation
\begin{equation}\label{sec1_ineq03}
\max_{0\le v\le 1}\frac{(1-v)+vt}{\left\{(1-v)+vt^{-1}\right\}^{-1}}=K(t),
\end{equation}
where $K(t):=\dfrac{(t+1)^2}{4t}$ is called Kantorovich constant \cite[Section 2]{FM2020}. 
\end{itemize}
From the relations \eqref{sec1_ineq02} and \eqref{sec1_ineq03}, we have some reverse inequalities for  \eqref{sec1_ineq01} as
$$
(1-v)+vt\le S(t)t^v,\quad t^v\le S(t)\left\{(1-v)+vt^{-1}\right\}^{-1},\quad (1-v)+vt\le K(t) \left\{(1-v)+vt^{-1}\right\}^{-1}
$$
 for  $0\le v\le 1$ and $t>0$.

In the paper \cite{PSMA2016}, the following relations were obtained:
\begin{equation}\label{sec2_ineq01}
t^v \le f_v(t) \le \frac{1}{2}\left\{t^v+(1-v)+vt\right\}\leq (1-v)+vt
\end{equation}
where for $0< v < 1$ and $t>0$ with $t\neq 1$
\begin{equation}\label{sec2_ineq01a}
f_v(t):= \dfrac{1}{\log t}\left\{\dfrac{1-v}{v}\left(t^v-1\right)+\dfrac{v}{1-v}\left(t-t^v\right)\right\}=L_{1/2}(t^v,1)\nabla_{\frac{v}{1-v}}L_{1/2}(t,1)
\end{equation}
 is the representing function of the weighted logarithmic mean 
$$L_v(a,b):=\dfrac{1}{\log a-\log b}\left\{\dfrac{1-v}{v}(a-a^{1-v}b^v)+\dfrac{v}{1-v}(a^{1-v}b^v-b)\right\},\quad (a,b>0,a\neq b)$$
with $L_v(a,a):=a$.
So, we easily find $L_v(1,t)=f_v(t)=t\cdot L_v(1/t,1)$ and the relation for four weighted means as
$$
a!_v b \le a\sharp_v b  \le L_v(a,b) \le a\nabla_v b.
$$
In addition, the further tight lower bound of $L_v(1,t) $ was given in \cite{FN2020}:
$$
t^{v}\leq t^{v/2}\nabla_v t^{(1+v)/2} \leq L_v(t,1). 
$$
Note that $L_{1/2}(a,b)=\dfrac{a-b}{\log a-\log b}$ for $a\neq b$ and so $L_{1/2}(1,t)=L_{1/2}(t,1)= f_{1/2}(t)=\dfrac{t-1}{\log t}$ for $t\neq 1$.

As generalized results from convex analysis, some extended results for the logarithmic mean have been obtained in \cite{RF1,RF2}.

In this paper, we give reverse inequalities for \eqref{sec2_ineq01}. Certain new inequalities for convex functions are also established. Using these, we derive refinements of arithmetic-logarithmic mean and Hermite-Hadamard inequalities.
 
\section{Weighted logarithmic mean}
We start this section with the following theorem.
\begin{theorem}\label{thm2.1}
For $0< v < 1$ and $t>0$ with $t\neq 1$, we have
\begin{itemize}
\item[(i)] $f_v(t) \le \left(\max\left\{ 1,1/t\right\}\cdot L_{1/2}(t,1)\right) t^v$.
\item[(ii)] $\dfrac{1}{2}\left\{t^v+(1-v)+vt \right\}\le \left(\dfrac{S(t)+1}{2}\right) f_v(t)$.
\end{itemize}
\end{theorem}

\begin{proof}
(i)~We have the following 
\begin{eqnarray*}
\frac{f_v(t)}{t^v} & = & \frac{1}{t^v \log t}\frac{(1-v)^2(t^v-1)+v^2(t-t^v)}{v(1-v)} \\
& = & \frac{1}{\log t} \frac{(1-2v)t^v+v^2t-(1-v)^2}{v(1-v)t^v} \\
& = & \frac{1}{\log t} \left\{ \frac{1-2v}{v(1-v)}+\frac{v}{1-v}t^{1-v}-\frac{1-v}{v}t^{-v} \right\} \\
& = & \frac{1}{\log t} \left\{ \left(\frac{1-v}{v}-\frac{v}{1-v} \right)+\frac{v}{1-v}t^{1-v}-\frac{1-v}{v}t^{-v} \right\} \\
& = & \frac{1}{\log t} \left\{ \frac{v}{1-v}(t^{1-v}-1)-\frac{1-v}{v}(t^{-v}-1) \right\} \\
& = & \frac{1}{\log t} \left\{ v \frac{t^{1-v}-1}{1-v}+(1-v)\frac{t^{-v}-1}{-v} \right\}. 
\end{eqnarray*}
Since $t^{1-v}-1 \geq \log t^{1-v} = (1-v) \log t$, 
\begin{equation}\label{sec2_ineq02}
\frac{t^{1-v}-1}{1-v} \geq \log t.
\end{equation}
Since $t^{-v}-1 \geq \log t^{-v} = -v \log t$, 
\begin{equation}\label{sec2_ineq03}
\frac{t^{-v}-1}{-v} \leq \log t.
\end{equation}
Since $t^{1-v}-1 \leq (1-v)t+v-1 = (1-v)(t-1)$, 
\begin{equation}\label{sec2_ineq04}
\frac{t^{1-v}-1}{1-v} \leq t-1.
\end{equation}
Since $t^{-v}-1=(t^{-1})^v-1 \leq vt^{-1}+1-v-1 = v(t^{-1}-1)$, 
\begin{equation}\label{sec2_ineq05}
\frac{t^{-v}-1}{-v} \geq 1-t^{-1}= 1-\frac{1}{t}.
\end{equation}
When $\log t > 0$, it follows from (\ref{sec2_ineq04}), (\ref{sec2_ineq03}) that 
$$
\frac{f_v(t)}{t^v} \leq \frac{1}{\log t} \{ v(t-1)+(1-v)\log t \} \leq \frac{t-1}{\log t}.
$$
When $\log t < 0$, it follows from (\ref{sec2_ineq02}), (\ref{sec2_ineq05}) that 
$$
\frac{f_v(t)}{t^v} \leq \frac{1}{\log t} \left\{ v \log t+(1-v)(1-\frac{1}{t}) \right\} \leq \frac{t-1}{t \log t}.
$$
Then we have the result. \\
(ii)~ Since $t^v \leq f_v(t)$, we have 
$$
\frac{(t^v+vt+1-v)/2}{f_v(t)} \leq \frac{(t^v+vt+1-v)/2}{t^v} = \frac{1}{2}+\frac{1}{2}\frac{\left\{ (1-v)+vt\right\}}{t^v} \leq \frac{1+S(t)}{2}.
$$
\end{proof}

Note that (i) and (ii) in Theorem \ref{thm2.1} are ratio type reverse inequalities for the first and second inequalities in \eqref{sec2_ineq01}, respectively. 

We also have the following interesting relations on the weighted logarithmic mean.
\begin{theorem}\label{theorem2.2}
For $0< v < 1$ and $t>0$ with $t\neq 1$, we have
$$
\min\left\{\frac{1-v}{v},\frac{v}{1-v} \right\}L_{1/2}(t,1)\le f_v(t) \le \max\left\{\frac{1-v}{v},\frac{v}{1-v} \right\}L_{1/2}(t,1).
$$
\end{theorem}

\begin{proof}
We have the following
\begin{eqnarray*}
f_v(t) & = & \frac{1}{\log t} \left\{ \frac{1-v}{v}(t^v-1)+\frac{v}{1-v}(t-t^v) \right\} \\
& = & \frac{1}{\log t} \left\{ \left( \frac{1-v}{v}-\frac{v}{1-v} \right)(t^v-1)+\frac{v}{1-v}(t-1) \right\} \\
& = & \frac{t-1}{\log t} \left\{ \left( \frac{1-v}{v}-\frac{v}{1-v} \right) \frac{t^v-1}{t-1}+\frac{v}{1-v} \right\} \\
& = & \frac{t-1}{\log t} \left\{ \frac{1-2v}{v(1-v)} \frac{t^v-1}{t-1}+\frac{v}{1-v} \right\}.
\end{eqnarray*}
We put $\displaystyle{F(t) = \frac{1-2v}{v(1-v)} \frac{t^v-1}{t-1}+\frac{v}{1-v}}$. Then 
\begin{eqnarray*}
F{'}(t) & = & \frac{1-2v}{v(1-v)} \frac{vt^{v-1}(t-1)-(t^v-1)}{(t-1)^2} \\
& = & \frac{1-2v}{v(1-v)} \frac{vt^v-vt^{v-1}-t^v+1}{(t-1)^2} \\
& = & \frac{1-2v}{v(1-v)} \frac{t^{v-1}(t^{1-v}-(1-v)t-v)}{(t-1)^2}.
\end{eqnarray*}
Since $(1-v)t+v \geq t^{1-v}$, $F^{'}(t) \leq 0$ for $0 < v < 1/2$.  And since 
\begin{eqnarray*}
\lim_{t \to \infty} F(t) & = & \frac{1-2v}{v(1-v)} \lim_{t \to \infty} \frac{t^v-1}{t-1}+\frac{v}{1-v} \\
& = & \frac{1-2v}{v(1-v)} \lim_{t \to \infty} vt^{v-1}+\frac{v}{1-v} \\
& = & \frac{1-2v}{v(1-v)} \lim_{t \to \infty} \frac{v}{t^{1-v}}+\frac{v}{1-v} = \frac{v}{1-v} > 0,
\end{eqnarray*}
we have $\displaystyle{\inf F(t) = \frac{v}{1-v}}$ and $\displaystyle{\sup F(t) = F(0) = \frac{1-v}{v}}$. \\
Since $(1-v)t+v \geq t^{1-v}$, $F^{'}(t) \geq 0$ for $1/2 < v < 1$.  Then we have 
$\displaystyle{\inf F(t) = \frac{1-v}{v}}$ and $\displaystyle{\sup F(t) = \lim_{t \to \infty}F(t) = \frac{v}{1-v}}$.
Thus 
$$
\min_t f_v(t) = \min \left\{\frac{1-v}{v}, \frac{v}{1-v} \right\} \frac{t-1}{\log t},  
$$
$$
\max_t f_v(t) = \max \left\{ \frac{1-v}{v}, \frac{v}{1-v} \right\} \frac{t-1}{\log t}.
$$
When $v = \frac{1}{2}, v= 0$ and $v = 1$, it is clear. Then we have the result. 
\end{proof}

At the end of this section, we give operator inequalities as consequences of our theorems.
Since $\int t^xdx = \frac{t^x}{\log t} +C$, the function $f_v(t)$ defined in \eqref{sec2_ineq01a} can be expressed by
$$
f_v(t) = \frac{1-v}{v}\int_0^vt^xdx+\frac{v}{1-v}\int_v^1t^xdx.
$$
For positive operators $A,B$, we define the weighted operator logarithmic mean as
$$
A\ell_v B:= A^{1/2}f_v\left(A^{-1/2}BA^{-1/2}\right)A^{1/2}=
\frac{1-v}{v}\int_0^v A\sharp_x Bdx +\frac{v}{1-v}\int_v^1 A\sharp_x Bdx\,\,\,\text{for}\,\,\,0<v<1.
$$
We see $A\ell_{1/2}B=\int_0^1 A\sharp_xB dx$ which is the operator logarithmic mean. We also use the standard notations for the weighted arithmetic mean $A\nabla_v B$, the weighted geometric mean $A\sharp_v B$ and the weighted harmonic mean $A!_vB$ for positive operators $A,B$ and $0\le v \le 1$:
$$
A\nabla_vB:=(1-v)A+vB,\,\,A\sharp_vB:=A^{1/2}\left(A^{-1/2}BA^{-1/2}\right)^vA^{1/2},\,\,A!_vB:=\left\{(1-v)A^{-1}+vB^{-1}\right\}^{-1},
$$
respectively. It was shown that \cite{PSMA2016}:
$$
A!_vB\le A\sharp_vB \le A\ell_v B \le \frac{1}{2}\left(A\sharp_vB+A\nabla_vB\right) \le A\nabla_vB.
$$
From Theorem \ref{theorem2.2}, we have the following inequalities.
\begin{corollary}
For $0< v < 1$ and positive operators $A,B$, we have
$$
\min\left\{\frac{1-v}{v},\frac{v}{1-v} \right\} A\ell_{1/2} B \le
A\ell_v B \le \max\left\{\frac{1-v}{v},\frac{v}{1-v} \right\} A\ell_{1/2} B.
$$ 
\end{corollary}
From Theorem \ref{thm2.1}, we also have the following results.
\begin{corollary}
For $0< v < 1$ and positive operators $A,B$ such that $\alpha A\leq B \leq \beta A$ with $0 <\alpha \le \beta$, we have
\begin{itemize}
\item[(i)] $A\ell_v B \le k_1\cdot A\sharp_v B$, where $k_1:=\max\limits_{\alpha\le t \le \beta}\left(\max\left\{1,1/t\right\}\cdot f_{1/2}(t)\right)$. 
\item[(ii)] $A\sharp_v B+A \nabla_v B \le k_2 \cdot A\ell_v B$, where $k_2:=\max\limits_{\alpha\le t \le \beta}\left(S(t)+1\right)$. 
\end{itemize}
\end{corollary}
If $\alpha \geq 1$, then $k_1=\dfrac{\beta-1}{\log \beta}$ and $k_2=S(\beta)+1$. If $\beta \le 1$, then $k_1=\dfrac{\alpha -1}{\alpha \log \alpha}$ and $k_2=S(\alpha)+1$. A few complicated analysis are possible for the constants $k_1$ and $k_2$ for the other cases. We omit them here.

\section{Convex functions approach}
This section gives upper and lower bounds on the gap in Jensen's inequality, i.e., 
\[\sum\limits_{i=1}^{n}{{{p}_{i}}f\left( {{a}_{i}} \right)}-f\left( \sum\limits_{i=1}^{n}{{{p}_{i}}{{a}_{i}}} \right)\]
where $f:J\subseteq \mathbb{R}\to \mathbb{R}$ is a convex function,  $a_i \in J$, and $p_i\geq 0$ with $\sum\limits_{i=1}^n p_i =1$. The first result contains several arithmetic means inequalities.
\begin{theorem}\label{prop3.1}
Let $f:J\subset \mathbb{R} \to [0,\infty)$ be a convex function and let $a_i \in J$, $p_i\geq 0$ and $\sum\limits_{i=1}^n p_i =1$. Then we have for any natural number $m$,
\begin{equation}\label{prop3.1_eq01}
\frac{A_f^m-G_f^m}{m A_f^{m-1}}\leq A_f-G_f\leq \frac{A_f^m-G_f^m}{m G_f^{m-1}},
\end{equation}
where
$$
A_f:=A_f(a_1,\cdots,a_n;p_1,\cdots,p_n):=\sum_{i=1}^n p_i f(a_i),$$
and
$$ G_f:=G_f(a_1,\cdots,a_n;p_1,\cdots,p_n):=f\left(\sum_{i=1}^np_ia_i\right).
$$
\end{theorem}

\begin{proof}
We use the identity $a^m-b^m=(a-b)(a^{m-1}+a^{m-2}b+\cdots +ab^{m-2}+b^{m-1})$. If $a \geq b > 0$, then we have
\begin{eqnarray*}
\frac{a^{m-1}+a^{m-2}b+\cdots +ab^{m-2}+b^{m-1}}{ma^{m-1}}
&\le& \frac{a^{m-1}+a^{m-2}b+\cdots +ab^{m-2}+b^{m-1}}{a^{m-1}+a^{m-2}b+\cdots +ab^{m-2}+b^{m-1}}\\
&\le& \frac{a^{m-1}+a^{m-2}b+\cdots +ab^{m-2}+b^{m-1}}{mb^{m-1}}.
\end{eqnarray*}
Multiplying $a-b \geq 0$ to both sides in the inequalities above, we have
\begin{equation}\label{prop3.1_eq02}
\frac{a^m-b^m}{ma^{m-1}}\leq a-b \leq \frac{a^m-b^m}{mb^{m-1}}.
\end{equation}
From Jensen inequality, we have $A_f\geq G_f$. Thus we get \eqref{prop3.1_eq01} by putting $a:=A_f$ and $b:=G_f$ in \eqref{prop3.1_eq02}.
\end{proof}

Applying Theorem \ref{prop3.1} to $m=2$, gives the following statement.
\begin{corollary}\label{8}
Let $f:[a,b]\to [0,\infty)$ be a convex function and let $0\leq v\leq 1$. Then
\[\begin{aligned}
  \frac{{{\left( f\left( a \right)\nabla_{v} f\left( b \right) \right)}^{2}}-{{f}^{2}}\left( a\nabla_{v} b \right)}{2\left(f\left( a \right)\nabla_{v} f\left( b \right)\right)} &\le f\left( a \right)\nabla_{v} f\left( b \right)-f\left( a\nabla_{v} b \right) \\ 
 & \le \frac{{{\left( f\left( a \right)\nabla_{v} f\left( b \right) \right)}^{2}}-{{f}^{2}}\left( a\nabla_{v} b \right)}{2f\left( a\nabla_{v} b \right)}.  
\end{aligned}\]
\end{corollary}

The well-known Hermite-Hadamard inequality states that
\begin{equation}\label{1}
f\left(\frac{a+b}{2}\right)\le\int_0^1 f(a\nabla_{v}b)dv \le \frac{f(a)+f(b)}{2} 
\end{equation}
for a convex function $f:[a,b]\to (0,\infty)$. 
By Corollary \ref{8}, we have a new interpolation of \eqref{1}.
\begin{corollary}\label{0}
Let $f:[a,b]\to (0,\infty)$ be a convex function. Then we have
\begin{eqnarray*}
 \int_0^1f(a\nabla_{v}b)dv-\int_0^1\frac{f^2(a\nabla_{v}b)}{f(a)\nabla_{v}f(b)}dv  &\le& 
\frac{f(a)+f(b)}{2}-\int_0^1f(a\nabla_{v}b)dv \\ 
&\le& \int_0^1\frac{\left(f(a)\nabla_{v}f(b)\right)^2}{f(a\nabla_{v}b)}dv-\frac{f(a)+f(b)}{2}.
\end{eqnarray*}
\end{corollary}
\begin{proof}
Taking integral over $0\le v \le 1$ in Corollary \ref{8}, we get
\[\begin{aligned}
   \frac{1}{2}\int_{0}^{1}{\left( f\left( a \right){{\nabla }_{v }}f\left( b \right) \right)dv }-\frac{1}{2}\int_{0}^{1}{\frac{{{f}^{2}}\left( a{{\nabla }_{v }}b \right)}{\left( f\left( a \right){{\nabla }_{v }}f\left( b \right) \right)}dv } 
 & \le \int_{0}^{1}{\left( f\left( a \right){{\nabla }_{v }}f\left( b \right) \right)dv }-\int_{0}^{1}{f\left( a{{\nabla }_{v }}b \right)dv } \\ 
 & \le \frac{1}{2}\int_{0}^{1}{\frac{{{\left( f\left( a \right){{\nabla }_{v }}f\left( b \right) \right)}^{2}}}{f\left( a{{\nabla }_{v }}b \right)}dv }-\frac{1}{2}\int_{0}^{1}{f\left( a{{\nabla }_{v }}b \right)dv }  
\end{aligned}\]
or equivalently,
\[\begin{aligned}
   \frac{1}{2}\frac{f\left( a \right)+f\left( b \right)}{2}-\frac{1}{2}\int_{0}^{1}{\frac{{{f}^{2}}\left( a{{\nabla }_{v }}b \right)}{\left( f\left( a \right){{\nabla }_{v }}f\left( b \right) \right)}dv } 
 & \le \frac{f\left( a \right)+f\left( b \right)}{2}-\int_{0}^{1}{f\left( a{{\nabla }_{v }}b \right)dv } \\ 
 & \le \frac{1}{2}\int_{0}^{1}{\frac{{{\left( f\left( a \right){{\nabla }_{v }}f\left( b \right) \right)}^{2}}}{f\left( a{{\nabla }_{v }}b \right)}dv }-\frac{1}{2}\int_{0}^{1}{f\left( a{{\nabla }_{v }}b \right)dv}.  
\end{aligned}\]
Adding $-\frac{1}{2}\frac{f(a)+f(b)}{2}$ and $\frac{1}{2}\int_0^1f(a\nabla_{v}b)dv$ to all sides in the above inequalities, then we have the desired inequalities. 
\end{proof}

\begin{remark}

From the first inequality in Corollary \ref{0}, with $f (a\nabla_{v}b) >0$ and $\dfrac{f(a\nabla_{v}b)}{f(a)\nabla_{v}f(b)}\leq 1$, we have
$$
\int_0^1 f(a\nabla_{v}b) dv \leq 2\int_0^1f(a\nabla_{v}b) dv -\int_0^1\frac{f^2(a\nabla_{v}b)}{f(a)\nabla_{v}f(b)}dv \le \frac{f(a)+f(b)}{2}.
$$
\end{remark}

The following result establishes an interpolation between the arithmetic and the logarithmic means.
\begin{theorem}\label{theorem3.2}
For $a,b>0$ and $0\le v \le 1$, we have
$$
L_{1/2}(a,b) \le \frac{1}{2}\left(a\nabla_{1/2} b +\int_0^1 \frac{\left(a\sharp_{v}b\right)^2}{a\nabla_{v}b}dv\right) \le a\nabla_{1/2} b.
$$
\end{theorem}
\begin{proof}
From the first inequality in Corollary \ref{0} we have
$$
2\int_{0}^{1}{f\left( a\nabla_{v} b \right)dv}\le \frac{f\left( a \right)+f\left( b \right)}{2}+\int_{0}^{1} \frac{f^2(a\nabla_{v}b)}{f(a)\nabla_{v}f(b)}dv.
$$
Now, replacing $a$ and $b$ by $\log a$ and $\log b$, respectively, we get
$$
2\int_0^1f(\log a^{1-v}b^{v})dv \le \frac{f(\log a)+f(\log b)}{2}+\int_{0}^{1} \frac{f^2(\log a^{1-v}b^{v})}{f(\log a)\nabla_{v}f(\log b)}dv.
$$
By choosing $f\left( v \right)=\exp v$, we have
$$
2\int_0^1 a^{1-v}b^{v}dv \le \frac{a+b}{2}+\int_0^1\frac{\left(a^{1-v}b^{v}\right)^2}{(1-v)a+v b}dv.
$$
Since $a^{1-v}b^{v} \le (1-v)a+v b$, we have
\begin{eqnarray*}
2 \left(\frac{b-a}{\log b-\log a}\right) &\le& \frac{a+b}{2}+\int_0^1\frac{\left(a^{1-v}b^{v}\right)^2}{(1-v)a+v b}dv\\
&\le& \frac{a+b}{2}+\int_0^1 \left((1-v)a+v b\right)dv =a+b
\end{eqnarray*}
which completes the proof.
\end{proof}


\begin{corollary}\label{cor_ref_amgm}
Let $x,y>0$. Then
\begin{equation}
\frac{(x-y)^2}{4(x+y)}\leq \frac{x+y}{2}-\sqrt{xy}\leq \frac{(x-y)^2}{8\sqrt{xy}}.
\end{equation}
\end{corollary}
\begin{proof}
Given $x,y>0,$ define the function $f:[0,1]\to [0,\infty)$ by $f(t)=x^{1-t}y^{t}.$ Letting $a=0,b=1$ in Corollary \ref{8}, we have
$$f(a)=x, f(b)=y, f\left(a\nabla_{v} b\right)=x\sharp_{v}y\;{\text{and}}\;f(a)\nabla_{v} f(b)=x\nabla_{v}y.$$
Substituting these values in Theorem \ref{8} with $v=\frac{1}{2}$ implies the desired inequalities.
\end{proof}

One can formulate a noncommutative version of Corollary \ref{cor_ref_amgm}.
\begin{corollary}
For positive operators $A,B$, we have
$$
\frac{1}{4}\left(B-3A\right) +\frac{1}{2} A !_{1/2}\left(AB^{-1}A\right)  \le A\nabla_{1/2}B-A\sharp_{1/2}B \le \frac{1}{8}\left(A\natural_{3/2}B-2A\sharp_{1/2}B+A\natural_{-1/2}B\right),
$$
where $A\natural_vB := A^{1/2}\left(A^{-1/2}BA^{-1/2}\right)^{v} A^{1/2}$ is defined for all real number $v$.
\end{corollary}

It is interesting that Corollary \ref{cor_ref_amgm} refines the well known inequality \cite{Mit1970}:
\begin{equation}\label{scalar_intro}
\frac{1}{8}\frac{(x-y)^2}{y}\leq \frac{x+y}{2}-\sqrt{xy}\leq \frac{1}{8}\frac{(x-y)^2}{x},
\end{equation}
valid for $x\leq y.$

The next theorem refines Theorem \ref{prop3.1}

\begin{theorem}\label{prop3.2}
Let $f:J\subset \mathbb{R} \to [0,\infty)$ be a nonlinear convex function and let $a_i \in J$, $p_i\geq 0$ and $\sum\limits_{i=1}^n p_i =1$. Then we have for any natural number $m$,
\begin{equation}\label{prop3.2_eq02}
\frac{(A_f^m-G_f^m)(A_f-\sqrt{A_fG_f})}{A_f^m-(\sqrt{A_fG_f})^m} \leq A_f - G_f \leq \frac{(A_f^m-G_f^m)(A_f-\sqrt{A_fG_f})}{(A_f+G_f-\sqrt{A_fG_f})^m-G_f^m}
\end{equation}
where
$A_f$ and $G_f$ are defined as in Theorem \ref{prop3.1}.
\end{theorem}

\begin{proof}
Let $a > b > 0$. We put 
$$
g(a,b) := a^{m-1}+a^{m-2}b+ \cdots +ab^{m-2}+b^{m-1} = \frac{a^m-b^m}{a-b} \geq 0.
$$
Note that $\dfrac{\partial g(a,b)}{\partial a} > 0$ for a fixed $b >0$, $\dfrac{\partial g(a,b)}{\partial b} > 0$ for a fixed $a >0$, and we have the inequalities 
$b \leq \sqrt{ab} \leq \frac{a+b}{2} \leq a+b-\sqrt{ab} \leq a$.
We consider the case $a > b$. Then we get 
$$
g(a+b-\sqrt{ab},b)< g(a,b) < g(a,\sqrt{ab}).
$$
That is,
$$
\frac{(a+b-\sqrt{ab})^m-b^m}{a-\sqrt{ab}} < g(a,b) < \frac{a^m-(\sqrt{ab})^m}{a-\sqrt{ab}}.
$$
Thus we have 
$$
\frac{g(a,b)(a-\sqrt{ab})}{a^m-(\sqrt{ab})^m} < \frac{g(a,b)}{g(a,b)} < \frac{g(a,b)(a-\sqrt{ab})}{(a+b-\sqrt{ab})^m-b^m}.
$$
Therefore we have 
\begin{equation}\label{prop3.2_eq03}
\frac{(a^m-b^m)(a-\sqrt{ab})}{a^m-(\sqrt{ab})^m} \leq a-b \leq \frac{(a^m-b^m)(a-\sqrt{ab})}{(a+b-\sqrt{ab})^m-b^m}.
\end{equation}
From Jensen inequality, we have $A_f > G_f$. Therefore we get (\ref{prop3.2_eq02}) by putting $a:= A_f$ and 
$b := G_f$ in (\ref{prop3.2_eq03}).
\end{proof}

We also show an alternative double inequality whose upper bound gives a refinement of that in Theorem \ref{prop3.1}.

\begin{theorem}\label{prop3.3}
Let $f:J\subset \mathbb{R} \to [0,\infty)$ be a nonlinear convex function and let $a_i \in J$, $p_i\geq 0$ and $\sum\limits_{i=1}^n p_i =1$. Then we have for any natural number $m$,
\begin{equation}\label{prop3.3_eq01}
\frac{m(A_fG_f)^{\frac{m-1}{2}}(A_f-G_f)^2}{A_f^m-G_f^m} \leq A_f - G_f \leq \frac{A_f^m-G_f^m}{m(A_fG_f)^{\frac{m-1}{2}}}
\end{equation}
where
$A_f$ and $G_f$ are defined as in Theorem \ref{prop3.1}.
\end{theorem}

\begin{proof}
Let $a > b > 0$. We put 
$$
g(a,b) := a^{m-1}+a^{m-2}b+ \cdots +ab^{m-2}+b^{m-1} = \frac{a^m-b^m}{a-b} \geq 0.
$$
We consider the case $a > b$. 
Since $g(a,b) > m(ab)^{\frac{m-1}{2}}$, we get 
$$
\frac{m(ab)^{\frac{m-1}{2}}}{g(a,b)} < \frac{g(a,b)}{g(a,b)} < \frac{g(a,b)}{m(ab)^{\frac{m-1}{2}}}.
$$
That is, 
$$
\frac{m(ab)^{\frac{m-1}{2}}(a-b)}{a^m-b^m} < \frac{g(a,b)}{g(a,b)} < \frac{a^m-b^m}{m(ab)^{\frac{m-1}{2}}(a-b)}.
$$
Therefore we have 
\begin{equation}\label{prop3.3_eq02}
\frac{m(ab)^{\frac{m-1}{2}}(a-b)^2}{a^m-b^m} \leq a-b \leq \frac{a^m-b^m}{m(ab)^{\frac{m-1}{2}}}.
\end{equation}
From Jensen inequality, we have $A_f > G_f$. Therefore we get (\ref{prop3.3_eq01}) by putting $a:= A_f$ and 
$b := G_f$ in (\ref{prop3.3_eq02}).
\end{proof}

\section*{Acknowledgements}
The authors would like to thank the referees for their careful and insightful comments to improve our manuscript.
This work was partially supported by JPSP KAKENHI grant numbers 16K05257 and 19K03525.


{\tiny (S. Furuichi) 
Department of Information Science, College of Humanities and Sciences, Nihon University, 3-25-40, Sakurajyousui, Setagaya-ku,
Tokyo, 156-8550, Japan}

{\tiny \textit{E-mail address:} furuichi@chs.nihon-u.ac.jp}

{\tiny (K. Yanagi) Department of Mathematics, Josai University, 1-1, Keyakidai, Sakado City, Saitama, 350-0295, Japan}

{\tiny \textit{E-mail address:} yanagi@josai.ac.jp}

{\tiny (H. R. Moradi) Department of Mathematics, Payame Noor University (PNU), P.O. Box 19395-4697, Tehran, Iran}

{\tiny \textit{E-mail address:} hrmoradi@mshdiau.ac.ir }

\end{document}